\documentclass[10pt,a4paper]{amsart}
\usepackage[utf8]{inputenc}
\usepackage{amsmath,amsthm}
\newtheorem{theorem}{Theorem}[section]
\usepackage{amsfonts}
\newtheorem{lm}[theorem]{Lemma}
\newtheorem{thm}[theorem]{Theorem}
\newtheorem{cor}[theorem]{Corollary}
\newtheorem{prop}[theorem]{Proposition}
\newcommand{\abs}[1]{\left\lvert #1 \right\rvert}
\def\F{\mathbb{F}}

\def\C{\mathbb{C}}
\def\Z{\mathbb{Z}}

\usepackage{amssymb}
\theoremstyle{definition}
\newtheorem{rem}[theorem]{Remark}

\title[A note on the set ${A(A+A)}$]{A note on the set
$\boldsymbol{A(A+A)}$}\author[P.-Y. Bienvenu]{Pierre-Yves Bienvenu}
\address{P.-Y. Bienvenu, Univ Lyon,  CNRS, ICJ UMR 5208, 
69622 Villeurbanne cedex, France}
\email{pbienvenu@math.univ-lyon1.fr}
\author[F. Hennecart]{Fran\c cois Hennecart}
\address{F. Hennecart, Univ Lyon, UJM-Saint-\'Etienne, CNRS, ICJ UMR 5208, 42023
Saint-\'Etienne, France}
\email{francois.hennecart@univ-st-etienne.fr}

\author[I. Shkredov]{Ilya Shkredov}
\address{I.D. Shkredov, Steklov Mathematical Institute, Division of Algebra and
Number Theory,
ul. Gubkina, 8, Moscow, Russia, 119991 and IITP RAS, Bolshoy Karetny per. 19,
Moscow, Russia,
127994}
\email{ilya.shkredov@gmail.com
}
\thanks{This work was performed within the framework of the LABEX MILYON
(ANR-10-LABX-0070) of Universit\'e de Lyon, within the program ``Investissements
d'Avenir" (ANR-11-IDEX-0007) operated by the French National Research Agency (ANR)}


\begin{document}
\begin{abstract}
Let $p$ a large enough prime number. When $A$ is a subset of
$\mathbb{F}_p\smallsetminus\{0\}$ of cardinality
$|A|> (p+1)/3$, then an application of Cauchy-Davenport Theorem gives
$\mathbb{F}_p\smallsetminus\{0\}\subset A(A+A)$. 
In this note,  we improve on this and we show that if  $|A|\ge 0.3051 p$ implies
$A(A+A)\supseteq\mathbb{F}_p\smallsetminus\{0\}$. In the opposite direction we show
that there exists a set $A$ such that
$|A| > (1/8+o(1))p$ and $\mathbb{F}_p\smallsetminus\{0\}\not\subseteq A(A+A)$.
\end{abstract}
\maketitle

\section{Introduction}

The aim of this note is to
study the size of the set $A(A+A)=\{a(b+c)\,:\, a,b,c\in A\}$ for a subset
$A\subseteq \F_p\smallsetminus\{0\}$.
This sort of problems belongs to the realm of expanding polynomials and
sum-product problems. In the literature, they are usually discussed in the
sparse set regime; for instance, Roche-Newton \textit{et al.} \cite{Roche},
Aksoy Yazici \textit{et al.} \cite{Aksoy}
proved that in the regime where $\abs{A}\ll p^{2/3}$,
one has
$\min(\abs{A+AA},\abs{A(A+A)})\gg |A|^{3/2}$ (see also  \cite{StevensZeuw}).
This implies in particular that as soon as $\abs{A}\gg p^{2/3}$, 
both sets  $A(A+A)$ and $A+AA$ occupy a positive proportion of $\F_p$.

Now we focus in the case where $A\subseteq \F_p$ occupies already a positive
proportion of $\F_p$. Let $\alpha=\abs{A}/p$, so we suppose
that
$\alpha >0$ is bounded below by a positive constant while
$p$ tends to infinity.
We will see that in this case the set
$A(A+A)$ contains all but a finite number of elements.
Besides, we prove that this finite number of elements may be strictly 
larger than one, unless
$\alpha$ is large enough.

\medskip
Here are our main results.
\begin{thm}
\label{Hennecart}
Let $A\subseteq \mathbb{F}_p$
 so that $|A|=\alpha p$ with $\alpha\ge 0.3051$. Then for any large enough prime
$p$, we
 have $A(A+A)\supseteq \mathbb{F}_p\smallsetminus\{0\}$.
\end{thm}

For smaller densities, we have the following result.
\begin{theorem}
\label{almostAll}
Let $A\subseteq \mathbb{F}_p\smallsetminus\{0\}$ and $0<\alpha<1$ satisfy
$\abs{A}\geq \alpha p$. Then one has 
$|A(A+A)| > p-1-\alpha^{-3}(1-\alpha)^2+o(1)$.
\end{theorem}

We note that similar results were obtained \cite{HH} for the set
$AA+A$. However, the constant $(1+c_0)^{-1}$
is replaced by the larger $1/3$ in Theorem \ref{Hennecart},
and the term $\alpha^{-3}(1-\alpha)^2$ is replaced by the larger
$\alpha^{-3}$. Further, the slightly weaker bound
$\abs{A(A+A)}\geq p-\alpha^{-3}$ may be extracted from 
\cite{sarko}.

In the opposite direction, we
have the following result. 
\begin{theorem}
\label{notquite}
There exists $A\subseteq \F_p\smallsetminus\{0\}$
such that $\abs{A}>(1/8+o(1))p$ and $A(A+A)\subsetneq \F_p\smallsetminus\{0\}$
for any large prime $p$.
Besides, for any $\epsilon >0$
there exists a set of size  $O(p^{3/4+\epsilon})$ such that
$A(A+A)$ misses $\Omega(p^{1/4-\epsilon})$ elements.
\end{theorem}

\section{Proof of Theorem \ref{Hennecart}}
In this section, we shall need the Cauchy-Davenport Theorem, which we
now state.
See for instance \cite[Theorem 2.2]{Nath} for a proof.
\begin{lm}
\label{CauchyDavenport}
Let $A$ and $B$ be subsets of $\F_p$. Then $\abs{A+B}\geq
 \min(\abs{A}+\abs{B}-1,p)$.
\end{lm}

In particular, if $\abs{A}+\abs{B}>p$, then $A+B=\F_p$, which is also
obvious because $A$ and $x-B$ cannnot be disjoint for any $x$.

First, we note that if $\alpha >1/2$, then $\abs{A+A}\geq \abs{A}>p/2$ so that
$A(A+A)=\F_p$. But as soon $\alpha <1/2$, we can easily
have $A(A+A)\subsetneq \F_p^*$, for instance
by taking $A=\{1,\ldots,\lfloor \frac{p-1}{2}\rfloor\}$.

Here is another almost equally immediate corollary.
\begin{cor}\label{cor2}
Let $A\subseteq \mathbb{F}_p\smallsetminus\{0\}$ satisfy $|A|>(p+1)/3$. Then either
$A(A+A)=\mathbb{F}_p$ or $\mathbb{F}_p\smallsetminus\{0\}$.
\end{cor}
\begin{proof}
Let $B=(A+A)\smallsetminus\{0\}$.
Using Theorem \ref{CauchyDavenport}, we have
 $|A+A|>(2p-1)/3$ so $|B|> (2p-4)/3$, whence
$|A|+|B|>p-1$. We infer that for any $x\in\mathbb{F}_p\smallsetminus\{0\}$ we have
$$
xB^{-1}\cap A\ne\varnothing,
$$
which yields $AB=\mathbb{F}_p\smallsetminus\{0\}$.
\end{proof}

We now prove Theorem \ref{Hennecart}, which reveals that we
can lower the density requirement from $1/3$ to $0.3051$
while maintaining 
$A(A+A)\supset \F_p\smallsetminus\{0\}$.

\medskip
To start with, we recall the famous Freiman's $3k-4$ Theorem  for the integers,
which gives precise structural information on a set which has quite
small, but not necessarily minimal,  doubling \cite[Theorem 1.16]{Nath}.

\begin{prop}\label{propAR} If $A\subset \mathbb{Z}$ satisfies $|A+A|\le 3|A|-4$ then
$A$ is contained in an 
  arithmetic progression of length at most $ |A+A|-|A|+1$.
\end{prop}

An analogue of this proposition has been developed in $\F_p$, and
it is known as the \emph{Freiman 2.4-theorem}. 
A useful lemma in \cite{Frei} (see also \cite[Theorem 2.9]{Nath})  was derived in
the proof thereof,
and we will need it here. We also include an improvement due to Lev. 
\\
We first define the Fourier transform of a function $f : \F_p\longrightarrow\C$
by 
$$
\widehat{f}(t)=\sum_{x\in\F_p}f(x)e_p(tx)
$$ 
for any $t\in\F_p$ where $e_p(x)=\exp(2i\pi x/p)$. The Parseval identity is the
following
\begin{equation}
\label{eq:ParsAdditive}
\sum_{x\in \mathbb{F}_p}f(x)\overline{g(x)}
=\frac1p\sum_{h\in \mathbb{F}_p}\widehat{f}(h)\overline{\widehat{g}(h)}.
\end{equation}

The characteristic function of a subset $A$ of $\mathbb{F}_p$ is denoted by $1_A$
and for $r\in \mathbb{F}_p$ we let $rA=\{ra,\ a\in A\}$. 

\begin{lm}
\label{Posnikova}
Let $A\subseteq \F_p$ with $|A|=\alpha p$ and $0<\gamma <1 $ satisfy
$\abs{\widehat{1_A}(r)}\geq \gamma\abs{A}$ for some $r\in\F_p\smallsetminus\{0\}$. 
Then there exists an interval modulo $p$ of length at most $p/2$ 
that contains at least $\alpha_1 p$ elements of 
$rA$ where $\alpha_1$ can be freely chosen as 
\begin{enumerate}
\item[i)]\label{Freiman} $\alpha_1 =\dfrac{(1+\gamma)\alpha}2$, \quad cf. \cite{Frei},

\item[ii)]\label{Lev} or $
\alpha_1=\dfrac{\alpha}2+\dfrac1{2\pi}\arcsin\left(
{\pi\gamma\alpha}
\right)$, \quad cf. \cite{Lev}.

\end{enumerate}
\end{lm}

There a few other basic results about Fourier transforms that we will
need in the sequel.
\begin{lm}
\label{FourierAP}
Let $P$ be an arithmetic progression in $\F_p$.
Then
$$
\sum_{r\in\F_p}\abs{\widehat{1_P}(r)}\ll p\log p.
$$
\end{lm}

We now recall Weil's bound \cite{Weil} for Kloosterman sums.
\begin{lm}\label{lmCompleteKloosterman}
For any $(a,b)\neq (0,0)$, we have
$$
\bigg|\sum_{k\in \F_p\smallsetminus\{0\}}e_p(ak+bk^{-1})\bigg|\leq 2\sqrt{p}
$$
\end{lm}

We will also need a bound for so-called incomplete Kloosterman sums,
whose proof follows easily from the last two lemmas.
\begin{lm}\label{lmKloosterman}
Let $P\subseteq \F_p\smallsetminus\{0\}$ be an arithmetic progression.
Then for any $r\neq 0$ we have
$$
\Big|\widehat{1_{P^{-1}}}(r)\Big|\ll \sqrt{p}\log p.
$$
\end{lm}

\medskip
Now we start the proof
of Theorem \ref{Hennecart}
itself. Let $\alpha\ge0.3051$,
$A\subseteq \mathbb{F}_p\smallsetminus\{0\}$ of size $|A|=\alpha p$ and denote
$B=(A+A)\smallsetminus\{0\}$.
We assume that there exists $x\in\mathbb{F}_p\smallsetminus\{0\}$ such that
$x\not\in A(A+A)$. Then
\begin{equation}\label{eqpart}
xB^{-1}\cap A=\varnothing,\quad (xA^{-1}-A)\cap A=\varnothing.
\end{equation}
It follows that $|A|+|B|\le p-1$, since otherwise
$AB=\mathbb{F}_p\smallsetminus\{0\}$. Hence
$|A+A|\le|B|+1\le p-|A|$. 

\medskip
We define
\begin{align*}
r_1(y)&=|\{(a,b)\in A\times A\,:\, y=xa^{-1}-b\}|,\\
r_2(y)&=|\{(c,d)\in A\times A\,:\, c+d\ne0\text{ and }y=x(c+d)^{-1}\}|,
\end{align*}
and $E_i=\sum_{y\in\mathbb{F}_p} r_i(y)^2$, $i=1,2$, the corresponding energies.
Observe from \eqref{eqpart} that
$$
\sum_{\substack{y\in\mathbb{F}_p\\r_1(y)+r_2(y)>0}}1\le p-|A|.
$$
By Cauchy-Schwarz we get
\begin{equation}\label{eq1}
4|A|^4=\Big(\sum_{y\in\mathbb{F}_p}\big(r_1(y)+r_2(y)\big)\Big)^2\le (p-|A|)\times
\sum_{y\in\mathbb{F}_p}\big(r_1(y)+r_2(y)\big)^2.
\end{equation}
Expanding the later inner sum gives
$$
\sum_{y\in\mathbb{F}_p}\big(r_1(y)+r_2(y)\big)^2=E_1+E_2+2\sum_{y\in\mathbb{F}_p}r_1(y)r_2(y).
$$ 
Let 
$$
\gamma=\max_{h\ne0}\frac{\big|\widehat{1_A}(h)\big|}{|A|}
$$
We have by Parseval
$$
pE_2=\sum_h\big|\widehat{1_A}(h)\big|^4=|A|^4+\sum_{h\ne0}\big|\widehat{1_A}(h)\big|^4\le
|A|^4+\gamma^2|A|^2(p|A|-|A|^2)
$$
and
\begin{align*}
pE_1&=\sum_h\big|\widehat{1_{xA^{-1}}}(h)\big|^2\big|\widehat{1_A}(h)\big|^2=|A|^4+\sum_{h\ne0}
\big|\widehat{1_{xA^{-1}}}(h)\big|^2\big|\widehat{1_A}(h)\big|^2\\
& \le |A|^4+\gamma^2|A|^2(p|A|-|A|^2).
\end{align*}
Moreover
\begin{align*}
p\sum_{y\in\mathbb{F}_p}r_1(y)r_2(y) &= \sum_h\widehat{1_{xA^{-1}}}(h)\widehat{1_A}(-h)
\widehat{r_2}(h)\\ &\le 
|A|^4 +\max_{h\ne0}|\widehat{r_2}(h)|
\sum_{h\ne0}
\big|\widehat{1_{xA^{-1}}}(h)\big|\big|\widehat{1_A}(h)\big|\\
&\le |A|^4 +\max_{h\ne0}|\widehat{r_2}(h)|(p|A|-|A|^2),
\end{align*}
by Parseval and Cauchy-Schwarz. For $h\ne 0$,
$$
\widehat{r_2}(h)  = \sum_{\substack{c,d\in A\\c+d\ne0}}e_p(hx(c+d)^{-1})
=
\frac1p\sum_{r}\sum_{z\ne0}\sum_{c,d\in A}e_p(r(c+d-z))e_p(hxz^{-1}),
$$
hence by the Parseval identity \eqref{eq:ParsAdditive} and Lemma
\ref{lmCompleteKloosterman}
$$
|\widehat{r_2}(h)|\le \frac1p\sum_{r}\big|\widehat{1_A}(r)\big|^2
\Big|
\sum_{z\ne0}e_p(hxz^{-1})
\Big|\ll \sqrt{p}|A|
$$
(similar arguments were used in \cite[Theorem 4]{Moshchevitin}).
We thus obtain from \eqref{eq1} and the above bounds
$$
2\alpha \le (1-\alpha)(2\alpha+\gamma^2(1-\alpha)+o(1)).
$$
This finally gives the lower bound
$$
\gamma\ge \frac{\sqrt{2}\alpha}{1-\alpha}+o(1).
$$
We are in position to apply Lemma \ref{Posnikova}, i). Let $A_1\subset A$ be such that 
$|A_1|\ge (1+\gamma)|A|/2$ and
$rA_1$ is included in an interval of length $p/2$ for some $r\ne0$. 
This shows that  $A_1$ is $2$-Freiman isomorphic\footnote{i.e. there exists a
bijection $f:A_1\longrightarrow A'_1$ such that $a+b=c+d\iff f(a)+f(b)=f(c)+f(d)$
for all $a,b,c,d\in A_1$.} to a subset $A'_1$ of $\Z$.
So we seek to apply Proposition \ref{propAR} to $A'_1$.
We get
\begin{equation}\label{eqalpha1}
\alpha_1=\frac{|A_1|}p\ge
f(\alpha)+o(1):=\frac{(1+(\sqrt2-1)\alpha)\alpha}{2(1-\alpha)}+o(1),
\end{equation}
and
\begin{equation}\label{eqc1}
c_1=\frac{|A_1+A_1|}{|A_1|}\le \frac{|A+A|}{|A_1|}\le \frac{(1-\alpha)p}{\alpha_1 p}
\le \frac{1-\alpha}{f(\alpha)}+o(1).
\end{equation}
In order to have $c_1<3$, it is sufficient to have that 
\begin{equation*}
\alpha>\frac{7-\sqrt{9+24\sqrt2}}{10-6\sqrt2}=0.29513\dots
\end{equation*}
which is satisfied since we have assumed $\alpha \geq 0.3051$.
We thus obtain that $A_1$ (resp. $A_1+A_1$) is contained inside an arithmetic 
progression 
$P_1$ (resp. $Q_1=P_1+P_1$) of length $|P_1|=|A_1+A_1|-|A_1|+1$ (resp. $2|P_1|-1$). 

\medskip
We define $B_1=(A_1+A_1)\smallsetminus\{0\}$ and $Q_1^*=Q_1\smallsetminus\{0\}$.
We need to estimate
$$
T=\frac1p\sum_{r\bmod{p}}\sum_{\substack{a\in P_1 \\ b\in Q_1^*}}e_p(r(a-b^{-1}x))
\ge\frac{|P_1||Q_1^*|}{p}-\frac1p \sum_{0<|r|<p/2}|\widehat{1_{P_1}}(r)|
|\widehat{1_{{Q_1^*}^{-1}}}(rx)|
$$ 
which counts the solutions $(a,b)\in P_1\times Q_1^*$
to the equation $a=b^{-1}x$.\\
Now $|\widehat{1_{P_1}}(r)|\ll {p}/{|r|}$ 
by Lemma
\ref{FourierAP}
and $|\widehat{1_{{Q_1^*}^{-1}}}(rx_0)|\ll \sqrt{p}\log p$
by Lemma \ref{lmKloosterman}
because $Q_1^*$ is the union of at most two arithmetic progressions.

As a result, we have
$$
T\ge\frac{|P_1||Q_1^*|}{p}+O(\sqrt{p}(\log p)^2). 
$$
The number of solutions to $a=b^{-1}x$ with $a\in P_1\smallsetminus A_1$ or $b\in
Q_1^*\smallsetminus B_1$ is at most $|P_1|-|A_1|+|Q_1^*|-|B_1|$. Since by assumption
there is no solution to $a=b^{-1}x$ with
$(a,b)\in A_1\times B_1$ we get
$$
T\le |P_1|-|A_1|+|Q_1^*|-|B_1|
$$
yielding
$$
\frac{|P_1||Q_1^*|}{p}\le |P_1|-|A_1|+|Q_1^*|-|B_1| +O(\sqrt{p}(\log p)^2).
$$
This implies
$$
\frac{(|B_1|-|A_1|)^2}p \le |B_1|-2|A_1|+O(\sqrt{p}(\log p)^2)
$$
whence
$$
\alpha_1(c_1-1)^2\le c_1-2+o(1).
$$
Because of \eqref{eqalpha1}, this gives 
\begin{equation}\label{eqeq}
f(\alpha)\times (c_1-1)^2-c_1+2\le o(1).
\end{equation}
The left-hand side of this inequality defines a function of $c_1$ which is
decreasing in the range $2\le c_1\le 1+1/(2f(\alpha))$. 
contradiction.
We check easily that $\alpha+f(\alpha)\ge1/2$ whenever 
$\alpha\ge 0.3$.
Hence for such $\alpha$
$$
\frac{1-\alpha}{f(\alpha)}\le 1+\frac1{2f(\alpha)}.
$$
We thus obtain from \eqref{eqc1} and \eqref{eqeq}
$$
f(\alpha)\left(\frac{1-\alpha}{f(\alpha)}-1\right)^2-\frac{1-\alpha}{f(\alpha)}+2\le
o(1),
$$
which reduces to
$$
(1-\alpha-f(\alpha))^2-(1-\alpha-2f(\alpha))\le o(1).
$$
In view of the definition of $f(\alpha)$ in \eqref{eqalpha1}, we get by expanding
the above formula
$$
(11-6\sqrt{2})\alpha^3-(22-6\sqrt{2})\alpha^2+17\alpha-4\le o(1),
$$ 
giving $\alpha <  0.305091+o(1)$, a contradiction for all $p$ large enough.
This concludes the proof of Theorem~\ref{Hennecart}.

\begin{rem} Using instead the sharpest result ii) of  Lemma \ref{Posnikova} leads to
a  slight improvement: 
if $|A|\ge 0.30065p$ then $\mathbb{F}_p\smallsetminus\{0\}\subseteq A(A+A)$ for any
large $p$.
The improvement is very small and uses non-algebraic expressions, which
is why we decided not to exploit it.

\end{rem}
\section{Proof of Theorem \ref{almostAll}}
We will now use multiplicative characters
of $\F_p$. We denote by $\mathfrak{X}$  the set of all multiplicative characters
modulo $p$ and by $\chi_0$ the trivial character.
In this context Parseval's identity is the statement that
\begin{equation}\label{eqpars}
\frac{1}{p-1}\sum_{\chi\in\mathfrak{X}}\bigg|{\sum_{x\in
\F_p\smallsetminus\{0\}}f(x)\chi(x)}\bigg|^2=\sum_{x\in
\F_p\smallsetminus\{0\}}\abs{f(x)}^2.
\end{equation}

We state and prove a lemma which is a multiplicative analogue
of a lemma of Vinogradov \cite{vino} (see also \cite[Lemma 7]{sarko}),
according to which
\begin{equation}
\label{eq:vinogradov}
\abs{\sum_{(x,y)\in A\times B}e_p(xy)}\leq \sqrt{p\abs{A}\abs{B}}.
\end{equation}

\begin{lm}
\label{characterSums}
For any subsets $A,B$ of $\F_p\smallsetminus\{0\}$
and any nontrivial character $\chi\in\mathfrak{X}$, we have
$$
\bigg|\sum_{(y,z)\in A\times B}\chi(y+z)\bigg|\leq
\left(\abs{A}\abs{B}p\right)^{1/2}\left(1-\frac{\abs{B}}p\right)^{1/2}
$$
\end{lm}

\medskip
We now prove Theorem \ref{almostAll}.
Let $A$ be a subset of $\F_p\smallsetminus\{0\}$ and $\alpha =\abs{A}/p$.
We estimate the number of nonzero elements in $A(A+A)$ by estimating
the number $N$ of solutions to
$$
x(y+z)=x'(y'+z')\neq 0,\quad x,y,z,x',y',z'\in A, 
$$
which we can rewrite as $x'x^{-1}(y+z)^{-1}(y'+z')=1$.
This number is
\begin{align*}
N &= \frac{1}{p-1}\sum_{\chi \in \mathfrak{X}}
\bigg|{\sum_{y,z\in A}\chi(z+y)\sum_{x\in A}\chi(x)}\bigg|^2\\
&\leq \frac{\abs{A}^6}{p-1}+\max_{\chi\neq \chi_0}\bigg|{\sum_{y,z\in A}
\chi(y+z)}\bigg|^2\times\frac{1}{p-1}\sum_{\chi\neq \chi_0}\bigg|{\sum_{x\in
A}\chi(x)}\bigg|^2,
\end{align*}
hence by Lemma \ref{characterSums} and Parseval's identity \eqref{eqpars}
\begin{align*}
N&\leq
\frac{\abs{A}^6}{p-1}+p\abs{A}^2(1-\alpha)\bigg(\abs{A}-\frac{\abs{A}^2}{p-1}\bigg)\\
&\leq \frac{\abs{A}^6}{p-1}+p\abs{A}^3(1-\alpha)^2\\
&\leq \frac{\abs{A}^6}{p-1}\Big(1+p^2\abs{A}^{-3}(1-\alpha)^2\Big)\\
&\leq \frac{\abs{A}^6}{p-1}\Big(1+p^{-1}\alpha^{-3}(1-\alpha)^2\Big).
\end{align*}
We let $\rho(w)=\abs{\{(x,y,z)\in A\times A \times A\mid w=x(y+z)\}}$ for
$w\in\mathbb{F}_p$.
Then
$$
N=\sum_{w\in A(A+A)\smallsetminus\{0\}}\rho(w)^2\quad\text{and}\quad
\sum_{w\in A(A+A)\smallsetminus\{0\}}\rho(w)\ge |A|^6-|A|^4.
$$
Finally  $N$ is related to $\abs{A(A+A)}$ by the Cauchy-Schwarz 
inequality as follows
\begin{align*}
\abs{A(A+A)}\geq |A(A+A)\smallsetminus\{0\}|&\geq (\abs{A}^6-\abs{A}^4) N^{-1}\\
&\geq
(p-1)\Big(1-\alpha^{-2}p^{-2}\Big)\Big(1+p^{-1}\alpha^{-3}(1-\alpha)^2\Big)^{-1}\\
&> p-1-\alpha^{-3}(1-\alpha)^2+o(1).
\end{align*}
This concludes the proof of Theorem \ref{almostAll}.\qed

\section{Proof of Theorem \ref{notquite}}
First we need a lemma.
\begin{lm}
\label{subRandom}
Let $c<1/2$ and $p$ large enough.
Let $P=\{1,\ldots, \lceil cp\rceil\}$.
Then the set $(P+P)^{-1}$ of the inverses (modulo $p$) of nonzero elements of $P+P$
has at most $2c^2p+O(\sqrt{p}(\log p)^2)$ common elements with $P$, that is,
$$\abs{(P+P)^{-1}\cap P}\leq 2c^2p+O(\sqrt{p}(\log p)^2).$$
\end{lm}
\begin{proof}
We note that
 $P+P=\{2,\ldots,2\lceil cp\rceil\}\subset \F_p\smallsetminus\{0\}$. 

Now we observe that
\begin{align*}
\abs{P\cap (P+P)^{-1}} &=\sum_{\substack{x\in P\\ y\in P+P\\
x=y^{-1}}}1=\frac{1}{p}\sum_{t\in \F_p}
\sum_{\substack{x\in P\\ y\in P+P}}e_p(t(x-y^{-1}))\\
&=
\frac{1}{p}\sum_{t\in \F_p}
\sum_{x\in P}e_p(tx)\sum_{y\in P+P}e_p(-ty^{-1})
\end{align*}
Using Lemmas \ref{FourierAP} and \ref{lmKloosterman}, 
we find
that
\begin{align*}
\abs{P\cap (P+P)^{-1}} &
=\frac{\abs{P}\abs{P+P}}p+\frac{1}{p}\sum_{t\in\mathbb{F}_p\smallsetminus\{0\}}\widehat{1_P}(t)\widehat{1_{(P+P)^{-1}}}(-t)\\
& =2c^2p+O(\sqrt{p}(\log p)^2).\qedhere
\end{align*}
\end{proof}

\bigskip
Now we prove Theorem \ref{notquite}.

\medskip
Let $c<1/2$ (to be determined later) and $p$ large enough.
Let $P=\{1,\ldots, \lceil cp\rceil\}$.
Let $A=P\smallsetminus (P+P)^{-1}$. It satisfies
$A\cap (A+A)^{-1}=\varnothing$, i.e. $1\neq A(A+A)$, and
has cardinality at least
$cp-2c^2p-O(\sqrt{p}(\log p)^2)$.
To optimise, we take $c=1/4$, in which case $\abs{A}\geq p/8-O(\sqrt{p}(\log p)^2)$.
For any $\epsilon>0$, for $p$ large enough, this is at least
$(1/8-\epsilon)p$, whence the first part of the theorem.

\medskip
For the second part, we note that
Lemma \ref{subRandom} provides a bound for the cardinality
$\abs{P\cap x(P+P)^{-1}}$ for any $x$,  so for any
$k\leq p-1$ we can get a set a of
size
$cp-2kc^2p-O(k\sqrt{p}(\log p)^2)$
so that
$A(A+A)$ misses $0$ and $k$ nonzero elements.
The main term is optimised for $c=1/4k$, where it is worth
$p/8k$.
Taking $k$ of size 
$p^{1/4}(\log p)^{-3/2}$, the error term is significantly smaller than the main term
(for large $p$), so we obtain a set $A$
of size  $\Omega(p^{3/4}(\log p)^{3/2})$ for which $A(A+A)$ misses at least
$p^{1/4}(\log p)^{-3/2}$ elements.
This is even a slightly stronger statement than claimed.\qed

\section{Final remarks}

\subsection{}
Let $p$ an odd prime, $a,b\in \mathbb{F}_p\smallsetminus\{0\}$ and assume that
$ba^{-1}=c^2$ is a square. 
Let $A\subset\mathbb{F}_p\smallsetminus\{0\}$. Then 
$a\not\in A(A+A)$ if and only if $b\not\in cA(cA+cA)=c^2A(A+A)$. Moreover $|cA|=|A|$.

We denote
$$
m_p=\max\big\{|A|\,:\, A\subseteq \mathbb{F}_p\smallsetminus\{0\}\text{ and
}A(A+A)\not\supseteq \mathbb{F}_p\smallsetminus\{0\}\big\}.
$$
From the above remark we have
$$
m_p=\max\{|A|\,:\, A\subseteq \mathbb{F}_p\smallsetminus\{0\}\text{ and }1\not\in
A(A+A)\text{ or }r\not\in A(A+A)\},
$$
where $r$ is any fixed nonsquare residue modulo $p$. 
By Theorems \ref{Hennecart} and \ref{notquite} we have
$$
3.277\dots \leq \liminf_{p\to\infty}\frac{p}{m_p}\leq
\limsup_{p\to\infty}\frac{p}{m_p} \le 8.
$$

\subsection{}
Let $p>3$ a prime number. The set $I$  of residues modulo $p$ in the interval
$\{r\in\mathbb{F}_p\,:\, p/3< r < 2p/3\}$ is  sum-free (i.e. $a+b\ne c$ for any
$a,b,c\in I$) and achieves the largest cardinality for those sets, namely
$|I|=\left\lfloor \frac{p+1}3\right\rfloor$, as it can be deduced from the
Cauchy-Davenport Theorem combined with the  fact
that $|I\cap(I+I)|=0$.

\medskip
 Let
$$
A=\{x\in I\,:\, x^{-1}\in I\}.
$$
Then $A=A^{-1}$ and $A$ is sum-free. It readily follows that $1\not\in A(A+A)$.
Moreover, since $I$ is an arithmetic progression, the events $x\in I$ and $x^{-1}\in
I$ are independent, so we may observe that $A$ has cardinality $\sim p/9$ as $p$
tends to infinity (it can be formally proved using Fourier analysis). This raises
the next question:\\[0.5em]
\textit{What is the largest size of a sum-free set
$A\subset\mathbb{F}_p\smallsetminus\{0\}$ such that $A=A^{-1}$ ?}\\[0.5em]
From Theorem \ref{Hennecart}, we deduce the following statement.

\begin{cor}\label{SFS}
Let $A\subset\mathbb{F}_p\smallsetminus\{0\}$ a sum-free set such that $A=A^{-1}$.
Then 
$|A|< 0.3051 p$ for any sufficiently large prime number $p$.
\end{cor}

This is related to the question of how large a sum-free multiplicative
subgroup of $\F_p^*$ can be. Alon and Bourgain showed \cite{AlonBourgain} that it
can be at least 
$\Omega(p^{1/3})$.

\subsection{}
Let $A\subset\mathbb{F}_p\smallsetminus\{0\}$ with $\alpha=|A|/p\gg1$, and let us
denote $A_s=A\cap(A+s)$. 
Let $0<\epsilon< 1$ be defined by 
$$
E^+(A)=\sum_{s\in A-A}|A_s|^2=(1-\epsilon)|A|^3,
$$
and $S$ be the subset of $A-A$ 
$$
S=\{s\in A-A\,:\, |A_s|>(1-\epsilon-p^{-1/3})|A|\}.
$$
Then
$$
E^+(A)\le (1-\epsilon-p^{-1/3})|A|\sum_{s\not\in
S}|A_s|+|A|^2|S|=(1-\epsilon-p^{-1/3})|A|^3+|A|^2|S|,
$$
from which we deduce 
\begin{equation}\label{eq|S|}
|S|\ge |A|p^{-1/3}.
\end{equation}

Assume that $A=A^{-1}$ and let $N$ be the number of solutions to equation
$$
(a-s)(b-t)=1,\quad (s,a,t,b)\in S\times A_s\times S\times A_t.
$$
For fixed $s,t\in S$, we have 
\begin{align*}
|(A-s)\cap(A_t-t)^{-1}|&=|A_s|+|A_t|-|(A-s)\cap(A_t-t)^{-1}|\\
&\ge 
2(1-\epsilon - o(1))|A|-|A|=(1-2\epsilon-o(1))|A|
\end{align*}
since $A_s-s\subset A$ and $(A_t-t)^{-1}\subset A^{-1}=A$. This yields
\begin{equation}\label{eqN}
N\ge (1-2\epsilon-o(1))|A||S|^2.
\end{equation}
On the other hand, denoting $r(x)=|\{(a,s)\in A\times S\,:\, x(a-s)=1\}|$, we have
$$
N \le \frac1p\sum_{h}\widehat{1_A}(h)\widehat{1_S}(-h)\widehat{r}(-h)\le
\frac{|A|^2|S|^2}{p}
+\max_{h\ne0}|\widehat{r}(h)|\times\frac1p\sum_{h}|\widehat{1_A}(h)\widehat{1_S}(h)|.
$$
By adapting equation \eqref{eq:vinogradov} we get $\max_{h\ne0}|\widehat{r}(h)|\leq \sqrt{p|A||S|}$ and by
Cauchy-Schwarz 
and Parseval we derive $N\le {|A|^2|S|^2}/{p}+O(\sqrt{p}|A||S|)$. Combined with
\eqref{eqN}, this gives
$$
\alpha + O({\sqrt{p}}{|S|^{-1}})\ge 1-2\epsilon-o(1),
$$
yielding by \eqref{eq|S|}, that $\epsilon\ge (1-\alpha)/2+o(1)$. 
Hence when $A=A^{-1}$
$$
E^+(A)\le \frac{1+\alpha+o(1)}2|A|^3.
$$
Together with Theorem \ref{Hennecart}, this implies the following result.
\begin{prop}
Let $A\subset \F_p^*$ be as in Corollary \ref{SFS}. Then for large enough $p$ the additive energy satisfies
$$
{E^+(A)}\le 0.6526{|A|^3}.
$$
\end{prop}

By considering similarly the multiplicative energy of $A$, it is possible to get the
following sum-product upper bound  for an arbitrary $A \subset \F_p$:  
$$
2E^+(A)+E^{\times}(A)\le (2+\alpha+o(1))|A|^3.
$$

%
%
%
%
%

\end{document}